\documentclass[letterpaper,10pt]{article}
\usepackage{amsmath}
\usepackage{amssymb}
\usepackage{amsthm}
\usepackage{amsfonts}

\usepackage{babel}
\usepackage{comment}
\usepackage{enumerate}
\usepackage{hyphenat}
\usepackage{graphicx} 

\usepackage[colorlinks=true, allcolors=blue]{hyperref}
\theoremstyle{plain}
\newtheorem{theorem}{Theorem}[section]

\theoremstyle{definition}
\newtheorem{example}[theorem]{Example}
\newtheorem{conjecture}[theorem]{Conjecture}

\newcommand{\G}{X}

\DeclareMathOperator{\Aut}{Aut}

\title{Edge-transitive token graphs as covers}
\author{Sergio G. G\'omez-Galicia, Octavio B. Zapata-Fonseca}
\date{May 2025}

\begin{document}
\maketitle

\begin{abstract}

This paper uses the theory of covering graphs  to characterize  some of the  edge-transitive graphs which can arise as token graphs. 

\end{abstract}

\section{Introduction}
\label{sec:intro}

All the edge-transitive graphs which can arise as token graphs 
have been classified by Zhang and Zhou in~\cite{ZZ}. 
The aim of the present paper is to characterize some of these edge-transitive token graphs as covering graphs. 
Before stating our main result precisely, Theorem~\ref{theo_1}, we will discuss the background to this result in both finite group theory and topological graph theory.
The graphs we consider will be directed or undirected, with or without loops, and with or without multiple edges. 
We will be specific. 

An \emph{automorphism} of a finite graph $X$ is a permutation of the vertex set $V(X)$ that preserves all the pairs of vertices belonging to the edges $E(X)$. 
These permutations of $V(X)$ form a group under composition, called the \emph{automorphism group} of $X$,  denoted by $\Aut(X)$. 
For the complete graph $K_n$ on $n$ vertices,  
$\Aut(K_n)$ is isomorphic to the symmetric group $S_n$ of degree $n$.  
If $K_{m,n}$ is the complete bipartite graph with $m$ independent vertices of degree $n$ and $n$ independent vertices of degree $m$,  then $\Aut(K_{m,n})$ is isomorphic to the Cartesian product $S_m \times S_n$ when $m\neq n$, and to the semidirect product $S_n^2 \rtimes \mathbb{Z}_2$ when $m = n$. 

A graph $X$ is said to be \emph{edge-transitive} if the action of $\Aut(X)$ on $E(X)$ has a single orbit. 
Such a graph may or may not be vertex-transitive. For example, the complete graph $K_n$ is vertex-transitive and edge-transitive. 
The complete bipartite graph $K_{m,n}$ is edge-transitive, and it is vertex-transitive only when $m=n$. 
Assuming that $X$ is a simple undirected graph without isolated vertices (i.e., all the vertices of $X$ have degree greater than one),  if $X$ is edge-transitive but not vertex-transitive, then the action of $\Aut(X)$ on $V(X)$ has exactly two orbits. 
These two orbits form a bipartition of $X$ (see, e.g.~\cite[Lemma 3.2.1]{AGT}), and so all graphs that are edge-transitive but not vertex-transitive are bipartite.

A bipartite graph such that all the vertices in the same part of the bipartition have the same degree is called \emph{biregular}. 
Every regular bipartite graph is biregular. 
The edge-transitive graphs which are not vertex-transitive are biregular. 
In this paper, we study certain edge-transitive biregular graphs which are not complete bipartite.  
Specifically,  we study the edge-transitive  graphs that arise from a construction known as the token graph.

Let $\G$ be a finite, undirected, and simple graph with $n$ vertices, and let $1 \le k \le n-1 $ be an integer. 
The \emph{$k$-token graph} $F_k(\G)$ of $\G$ is the graph 
with the set of all $k$-subsets of $V(X)$ as vertex set, where two vertices are adjacent when their symmetric difference is an edge of $\G$. 
Token graphs were introduced by Fabila-Monroy, Flores-Peñaloza, Huemer, Hurtado, Urrutia, and Wood~\cite{FFHH}. 
However, these graphs have been independently introduced several times under different names. See, e.g.,~\cite{John,ABE,AGR}.

It is easy to see that $F_{k}(\G) \cong F_{n-k}(\G)$. 
Therefore we always assume $1 \leq k \leq n/2$. 
Moreover, throughout the rest of the paper, we assume 
$V(X)  = \{1, \dots, n\}$. 
Thus, 
\[
|V(F_{k}(\G))| =\binom{n}{k}. 
\]
The $k$-token graph $F_k(K_n)$ of the complete graph $K_n$ is isomorphic to the \emph{Johnson graph} $J(n,k)$, defined as the graph with vertex set all the $k$ element subsets of a set with $n$ elements, where two vertices of $J(n,k)$ are adjacent whenever they intersect in exactly $k-1$ elements. 
Johnson graphs are examples of graphs with  a vertex-transitive and edge-transitive group of automorphisms which are not complete.

As we mentioned before, our intention is to characterize all the edge-transitive token graphs  from the following classification~\cite[Theorem 1.2]{ZZ}. 

\begin{theorem}[\cite{ZZ}]
Let $X$ be a connected graph. 
The token graph $F_k(X)$  is edge-transitive if and only if one of the following holds:
\begin{enumerate}[(1)]
    \item $X\cong K_n$ and $2\leq k \leq n-1$;
    \item $X\cong K_{1,n}$ and $2\leq k \leq n$;
    \item $X\cong K_{2,n}$ and $k = (n+2)/2$;
    \item $X\cong K_{n,n}$ and $k = 2$ or $k=2(n-1)$.
\end{enumerate}
\end{theorem}

We note that the family of edge-transitive token graphs  $F_k(K_{1,n})$ contains concrete examples of edge-transitive biregular graphs.
For instance, the token graphs $F_k(K_{1,n})$ for $n$ odd are the doubled Johnson graphs, and the token graphs $F_2(K_{1,n})$ are the barycentric subdivisions of the complete graphs $K_n$.

\section{Combined base graphs}
Edge-transitive biregular graphs have been extensively studied by the theory of covering graphs. 
We now introduce the notion of combined base graphs, and give a way to construct covering graphs above them. 
For edge-transitive biregular graphs, this construction was first studied under the name \emph{graph of the completion} of an amalgam in~\cite{Gold}. 
It has also been studied under the names \emph{coset graph} in~\cite{GLP},  \emph{double coset graph} in~\cite{DX}, and \emph{generalized cover graph} in~\cite{PT}. 
Covering graphs are defined in many texts, see e.g.,~\cite{biggs,AGT,GT,sunada}.

Let $G$ be a group, and let $H < G$ be a subgroup. The set of right cosets of $H$ in $G$ is denoted by $G / H$, and the index of $H$ in $G$ is denoted by $[G:H]$.

A \emph{combined voltage assignment} on $\G$ in a group $G$ is a pair of functions $(\alpha,\omega)$, where $\alpha$ is a voltage assignment $\alpha:E(\G)\to G$, and $\omega$ is a map that sends each vertex $x\in V(\G)$ to a subgroup $\omega(x)<G$. 
The graph $\G$ is called a \emph{combined base graph} if it is equipped with the combined voltage assignment $(\alpha,\omega)$ for some group $G$.

The \emph{covering graph} $\G^{(\alpha,\omega)}$ of the combined base graph $\G$ with respect to the group $G$ is the graph with vertex set
\[V^{(\alpha,\omega)}:=\{(x,K) \mid x\in V(\G) \text{ and } K\in G/\omega(x)\},\] 
and the edges of $\G^{(\alpha,\omega)}$ defined as follows: for each edge $xy\in E(\G)$ with an assigned voltage $\alpha(xy)\in G$, there is an edge in the covering graph $\G^{(\alpha,\omega)}$ from a vertex $(x,K)$, for some $K\in G/\omega(x)$, to a vertex $(y,H)$, for some $H\in G/\omega(y)$ if and only if $K\alpha(xy) \cap H\neq \emptyset$. 
If $\omega$ maps each vertex $x\in V(X)$ to the trivial group $\{0\}$, then $\omega$ is the trivial map $0$, the base graph $X$ is the so-called \emph{voltage graph}, and the covering graph $X^{(\alpha,0)}$ is called the covering lift (sometimes called derived graph) of the voltage graph $X$. 

\section{The token graph $F_2({K_n})$ as a covering graph}
\label{sec:F2}

There are several graphs that can be obtained as covering graphs of combined base graphs. 
For example, it is well known that Cayley graphs can be obtained as covers of base graphs with a single vertex incident with multiple loops (see \cite{GT}). 
Another interesting example is the line graph $L(K_n)$ of the complete graph $K_{n}$, also known as the triangular graph $T(n)$.
This graph is a strongly regular graph with parameters $(\binom{n}{2},2(n-2),n-2,4)$ if $n\geq 4$. 
For $m = 5$ it is the complement of the Petersen graph.
There is an isomorphism between the strongly regular graph $T(n)\cong L(K_{n})$ and $F_2({K_n})$ given by the mapping $ij\mapsto \{i,j\}$.

Dalf{\'o}, Fiol, Pavl{\'\i}kov{\'a}, and Sir{\'a}n  \cite{DFP} proved that for $n$ odd, the graph $T(n) \cong F_2({K_n})$ can be obtained as a cover, and also mentioned, without proving, that the token graph $F_2(K_n)$ with $n$ even can be obtained as a cover.

In this paper we construct a combined base graph, and give an isomorphism between its covering graph and the token graph $F_2(K_n)$ for $n$ even. 

\begin{theorem} \label{theo_1}
Let $n$ be an even number. 
The token graph $F_2(K_n)$ is isomorphic to the covering graph $\G^{(\alpha,\omega)}$ of a combined base graph $\G$ on $n/2$ vertices and with combined assignment on the group $\mathbb{Z}_{n}$.
\end{theorem}
\begin{proof}
We construct the combined based graph $\G$ on the group $\mathbb{Z}_{n}$ as follows. 
Since $n$ is even, $n/2$ is an integer. 
Let the $n/2$ vertices of $X$ be enumerated as
\[
x_1,x_2,\dots,x_{n/2}.
\]
The map $\omega$ that sends each vertex $x_i$ to a subgroup of $\mathbb{Z}_n$  is defined by 
\[
	\omega(x_i):=
	\begin{cases}
        0\mathbb{Z}_n=\{0\} & \textnormal{ if }  1\leq i < n/2, \\
		i\mathbb{Z}_n =\{0,i\} & \textnormal{ if }  i=n/2.
	\end{cases}
  \]  
 We remark that $\G$ contains multiple edges and loops. Now we define the voltage assignment $\alpha$ in $\G$: 
\begin{itemize}
    \item If $1\leq i,j\leq n/2$, $i\neq j$, then there is an edge $e_{i,j}:=x_ix_j$ such that $\alpha(e_{i,j})=0$.
    \item If $1\leq i < n/2$, then $x_i$ has a loop $x_ix_i$.
    \item If $1\leq i < n/2$ and $j>i$, then
    \begin{itemize}
        \item Each vertex $x_i$ has a loop $x_ix_i$ such that $\alpha(x_ix_i)=1$.
        \item There are the following edges: $e'_{i,j}=x_ix_j$, $f_{i,j}=x_ix_j$ and $f'_{i,j}=x_ix_j$, such that $\alpha(e'_{i,j})=i$, $\alpha(f_{i,j})=n-j+i$ and $\alpha (f'_{i,j})=n-j$.
    \end{itemize}
\end{itemize}
 
 Now we show that the covering graph $\G^{(\alpha,\omega)}$ is isomorphic to $F_{2}(K_n)$. 
 First observe that 
 \[|V^{(\alpha,\omega)}|=|V(F_2(K_n))|,\]
 since the number of vertices of $\G^{(\alpha,\omega)}$ is given by
\[
\sum_{i=1}^{n/2}[\mathbb{Z}_n:\omega(x_i)]= \sum_{i=1}^{n/2-1}[\mathbb{Z}_n:\omega(x_i)]+ [\mathbb{Z}_n:\omega(x_{n/2})]=\left(\frac{n}{2}-1\right)n+\frac{n}{2}=\binom{n}{2}.
\]

Now we define the map $\varphi$ from $\G^{(\alpha,\omega)}$ to $F_{2}({K_n})$ as follows.
\begin{itemize}
    \item If $ 1\leq i< n/2$, then
\[
\varphi\left((x_i,\{j\})\in V^{(\alpha,\omega)}\right)= \{1+j\mod{n},1+j+i \mod{n}\}%\in V(F_2(K_n)).%\quad \text{ for all } 0\leq j\leq n-1.
\]
for all $0\leq j\leq n-1$.
    \item If $i=n/2$, then 
 \[
\varphi\left((x_{i},\{j,n/2+j\}\in V^{(\alpha,\omega)}\right)= \{1+j \mod{n},1+j+n/2 \mod{n}\}%\in V(F_2(K_n)). % \quad \text{ for all } 0\leq j\leq n-1.
 \]
 for all  $0\leq j\leq n-1$.
\end{itemize}

 By construction of the combined voltage assignment $(\alpha, \omega)$ on the base graph $\G$, it can be seen that the map $\varphi$ defined above is an isomorphism between the covering graph $\G^{(\alpha,\omega)}$ and the token graph $F_2(K_n)$.
\end{proof}
\begin{figure}[h!]
        \centering
        \includegraphics[width=0.85\linewidth]{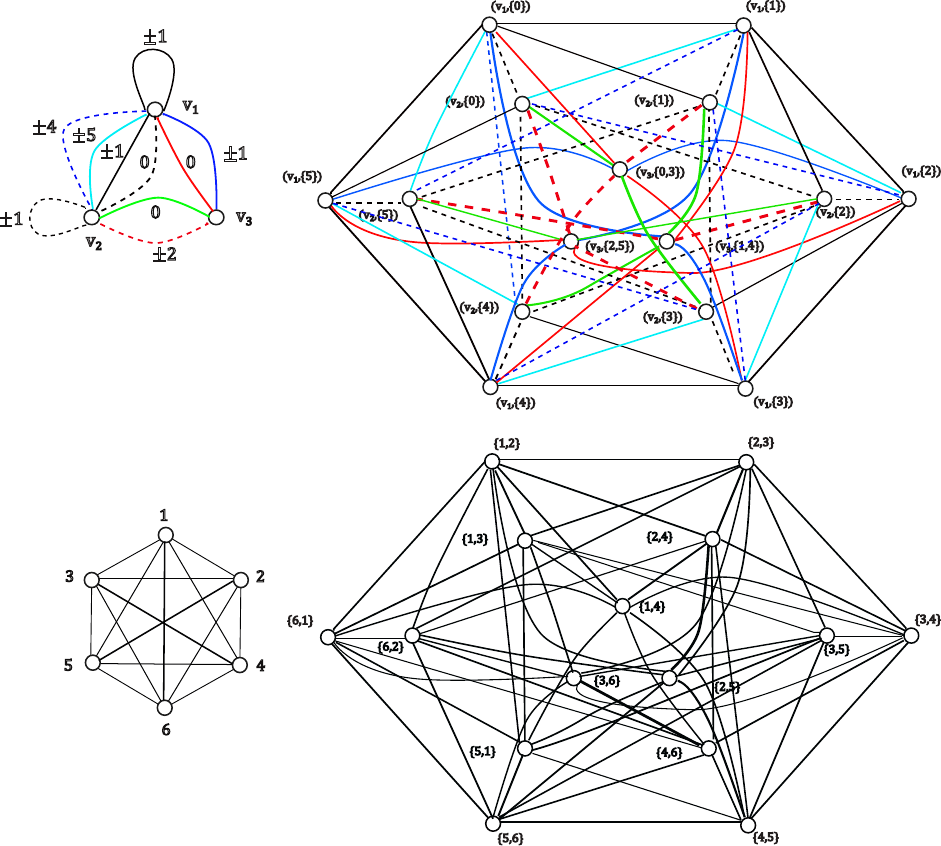}
         \caption{A combined based graph with cover isomorphic to $F_{2}(K_6)$}
        \label{fig:ex_1}
    \end{figure}
\begin{example}
Consider a combined base graph $\G$ with a combined voltage assignment $(\alpha,\omega)$ in $\mathbb{Z}_{6}$ described as in Theorem \ref{theo_1}. Then the covering graph $\G^{(\alpha,\omega)}$ is isomorphic to $F_2(K_6)$.
 The combined base graph $\G$ is shown in the upper left part of Figure \ref{fig:ex_1}, and to the right its covering graph $\G^{(\alpha,\omega)}$ is shown;  in the lower left part we can see the complete graph $K_6$, and to the right we can see its $2$-token graph.
\end{example}

\section{The token graph $F_{k}(K_{1,n})$ as a covering graph}

In \cite{DDF} it was observed that the $k$-token graph $F_k(K_{1,n})$ of the star graph $K_{1,n}$ when $k=(n+1)/2$ (and so, $n$ is odd) is a particular case of the \emph{doubled Johnson graphs}. It should be noted that this family of graphs is relevant since they are distance biregular graphs which are not bipartite distance regular (see, e.g.~\cite{Delorme,GS}).

Now we show that, for some particular cases, the graph $F_{k}(K_{1,n})$ can be obtained as a covering graph of a combined base graph with trivial map $\omega = 0$.
When $n=3$, it is straightforward to see that $F_{k}(K_{1,n})$ is isomorphic to the cycle $C_6$, and thus it can be obtained as a cover of a so-called voltage graph (since any cycle is a Cayley graph). 
We can extend this result to the case when $n=5$ and $k=3$. 
In the upper left part of Figure \ref{fig:ex_2} we can see the star $K_{1,5}$; in the lower left part it is shown a base graph; finally on the right side, it can be observed the $3$-token graph of $K_{1,5}$ seen as both a token graph and a cover of the base graph. 
\begin{figure}[htp!]
        \centering
     \includegraphics[width=0.7\linewidth]{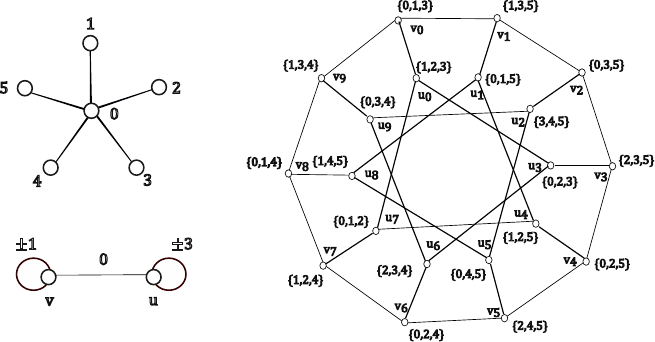}
         \caption{The star $K_{1,5}$, its $3$-token graph, and the base graph whose cover is isomorphic to $F_3(K_{1,5})$}.
        \label{fig:ex_2}
\end{figure}
\begin{conjecture}
\label{conj:1}
Let $n$ be an odd number and $k = (n+1)/2$. 
The token graph $F_k(K_{1,n})$ is isomorphic to the covering graph $\G^{(\alpha, \omega)}$ of a base graph $\G$ on $\frac{1}{2n}\binom{2k}{k}$ vertices and with combined assignment on the group $\mathbb{Z}_{2n}$.
\end{conjecture}

\section{The token graph $F_2(K_{1,n})$ as a covering graph}

Now we focus on the case of the $2$-token graph of the star graph $K_{1,n}$. 
In~\cite{GT} it was observed that this graph is isomorphic $K_n$ with exactly one subdivision on each edge. The isomorphism is given by mapping $i\to \{0,j\}$ and $(ij)\to \{i,j\}$, where $0$ is the center of the star $K_{1,n}$, and $(ij)$ is the vertex obtained in the subdivision of the edge $ij$.
 In Figure \ref{fig:ex_3}, we show a base graph whose cover is isomorphic to $F_2(K_{1,5})$, which is shown on the right as both a token graph and a cover.
    \begin{figure}[htp!]
        \centering
      \includegraphics[width=0.55\linewidth]{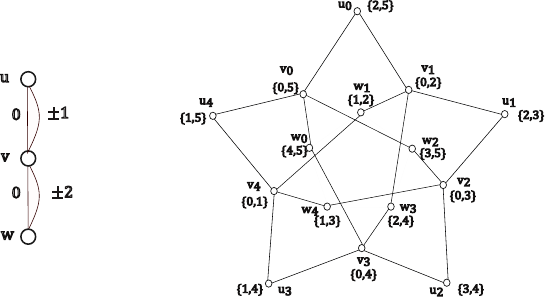}
         \caption{The star $K_{1,5}$, its $2$-token graph, and the combined graph with cover isomorphic to $F_2(K_{1,5})$}.
        \label{fig:ex_3}
    \end{figure}
We have the following conjecture which summarizes the discussion above.

\begin{conjecture}
 \label{conj:2}
 Let $n$ be a number such that $n$ divides $\binom{n+1}{2}$.
The token graph $F_2(K_{1,n})$ is isomorphic to the covering graph $\G^{(\alpha, \omega)}$ of a base graph $\G$ on $n-k$ vertices and with combined assignment on the group $\mathbb{Z}_{n}$.
\end{conjecture}

\section{Some conclusions and future work}

It is well-known in the theory of covering graphs, that a  graph $X$ can be derived as a covering graph from a smaller base graph with assignment $(\alpha, 0)$ if and only if there exists a nontrivial free action of $\Aut(X)$ on the vertices of $X$ (see, e.g.~\cite[Theorem 2.2.2]{GT}).

Since the token graph $F_k(K_{1,n})$ with $k = (n + 1)/2$ is isomorphic to the doubled Johnson graph,
and  the token graph $F_2(K_{1,n})$ is isomorphic to the barycentric subdivision of the complete graph $K_n$, an immediate consequence of our conjectures, if true, would be that the automorphism groups of the doubled Johnson graphs and the subdivision of $K_{n}$ always have a nontrivial free action on the vertices, and we know  the quotient graphs. 

\bibliographystyle{elsarticle-num}

\end{document}